\documentclass[11pt]{article}

\usepackage{subfigure}
\usepackage[english]{babel}

\usepackage[center]{caption2}
\usepackage{amsfonts,amssymb,amsmath,latexsym,amsthm}
\usepackage{multirow}
\usepackage{cite}
\usepackage{color, xcolor}

\def\ex{\mbox{ex}}

\newtheorem{thm}{Theorem}[section]
\newtheorem{cor}[thm]{Corollary}
\newtheorem{lem}[thm]{Lemma}
\newtheorem{prop}[thm]{Proposition}

\newtheorem{definition}{Definition}

\newtheorem{question}{Question}

\begin{document}	
	\title{Generalized Tur\'an number with given size}
        \author{Yan Wang, \quad Yue Xu, \quad Jiasheng Zeng, \quad Xiao-Dong Zhang{\footnote{\small Corresponding author: Xiao-Dong Zhang (E-mail: xiaodong@sjtu.edu.cn)}}\\
		\small  School of Mathematical Sciences, MOE-LSC, SHL-MAC and CMA-Shanghai\\
		\small Shanghai Jiao Tong University\\
		\small  800 Dongchuan Road, Shanghai 200240, China
	}
   	\date{}
	\date{}
	
	\maketitle
	
	\begin{abstract}
        Generalized Tur\'an problem with given size, denoted as $\mathrm{mex}(m,K_r,F)$, determines the maximum number of $K_r$-copies in an $F$-free graph with $m$ edges. We prove that for $r\ge 3$ and $\alpha\in(\frac 2 r,1]$, any graph $G$ with $m$ edges and $\Omega(m^{\frac{\alpha r}{2}})$ $K_r$-copies has a subgraph of order $n_0=\Omega(m^\frac{\alpha}{2})$, which contains $\Omega(n_0^{\frac{i(r-2)\alpha}{(2-\alpha)r-2}})$ $K_i$-copies for each $i = 2, \ldots, r$. This implies an upper bound of $\mathrm{mex}(m, K_r, F)$ when an upper bound of $\mathrm{ex}(n,K_r,F)$ is known. Furthermore, we establish an improved upper bound of $\mathrm{mex}(m, K_r, F)$ by $\mathrm{ex}(n, F)$ and $\min_{v_0 \in V(F)} \mathrm{ex}(n, K_r, F - v_0)$.      
        As a corollary, we show $\mathrm{mex}(m, K_r, K_{s,t}) = \Theta( m^{\frac{rs - \binom{r}{2}}{2s-1}} )$ for $r \geq 3$, $s \geq 2r-2$ and $t \geq (s-1)! + 1$, and obtain non-trivial bounds for other graph classes such as complete $r$-partite graphs and $K_s \vee C_\ell$, etc. 
	\end{abstract}


	
\section{Introduction}
In this paper, we only consider simple and finite graphs. For a graph $G=\left( V,E\right)$ and a vertex $v\in V$, the degree of $v$ is denoted by $d_G\left( v\right)$. The maximum and minimum degrees of $G$ are  denoted by $\Delta(G)$ and $\delta(G)$ respectively. The numbers of vertices and edges of $G$ are denoted by $v\left( G\right)$ and $e\left( G\right)$ respectively, which are also called the order and the size of $G$, respectively. Let $K_r$ denote the complete graph on $r$ vertices. We denote by $\mathcal{N}(F,G)$ the number of $F$-copies in $G$, and we write $k_r(G)=\mathcal{N}(K_r,G)$. Let $f$ and $g$ be two positive-valued functions of $t$. We use $f = o(g)$ to mean $\lim_{t \to \infty} \frac{f}{g} = 0$; we use $f = O(g)$, $f=\Theta(g)$, and $f = \Omega(g)$ to mean that there exist real constants $C > 0,C_2\ge C_1>0$ such that $\limsup_{t \to \infty} \frac{f}{g} \leq C$, $C_1\le \liminf_{t \to \infty} \frac{f}{g}\le\limsup_{t \to \infty} \frac{f}{g} \le C_2$, and $\liminf_{t \to \infty} \frac{f}{g} \geq C$, respectively. 

For a graph $H$, the \textit{Tur\'an number} $\text{ex}(n, H)$ is the maximum number of edges in an $n$-vertex graph that does not contain $H$ as a subgraph. Determining Tur\'an numbers for various graphs and hypergraphs is a central problem in extremal combinatorics and has been studied extensively. Alon and Shikhelman~\cite{Alon_Many T copies in H-free graphs} initiated the systematic study of a generalized version of this problem. For two graphs $F$ and $H$, the \textit{generalized Tur\'an number} $\text{ex}(n, F, H)$ is the maximum number of $F$-copies in an $n$-vertex $H$-free graph. Note that $\text{ex}(n, H)$ is the case where $F$ is a single edge. Moreover, Alon and Shikhelman~\cite{Alon_Many T copies in H-free graphs} proved that if the chromatic number $\chi(H)$ of $H$ is greater than $r$, then $\text{ex}(n, K_r, H) = \Theta(n^r)$. The problem of estimating generalized Tur\'an numbers has recently attracted significant attention, and many classical results have been generalized to this setting. For detailed discussions and progress, see~\cite{complete_r_partite_free, Bollobas_Intro, Chase_Intro, Gerbner_Intro, Grzesik_Intro, Hatami_Intro, Luo_Intro, Ma_Intro, Zhu_Intro, Zykov_Intro}, and a comprehensive survey by Gerbner and Palmer~\cite{Survey_Gerbner}.

A natural variation of this problem is to investigate the cases when the number of edges in the graph is given instead of the number of vertices. Let $\text{mex}(m, H, F)$ denote the maximum number of $H$-copies in an $F$-free graph with $m$ edges. It was first systematically proposed by Radcliffe and Uzzell in~\cite{18_Radcliff_Intro_edge}, and this article also investigated the stability and Erdős-Stone type results for $\text{mex}$. Additionally, the function $\text{mex}(m, H, F)$ has been determined for the cases when $H = K_t$ and $F$ is a path~\cite{4_Chakraborti_Intro_edge}, a star~\cite{3_Chakraborti_Intro_edge, 14_Kirsch_Intro_edge}, or a clique~\cite{Eckhoff_Intro_edge,9_Frohmader_Intro_edge}.

In this paper, we mainly focus on the generalized Tur\'{a}n problem $\mathrm{mex}(m,H,F)$ for $H=K_r$. We start by considering the number of triangles in a graph. For a graph $G$ with $n$ vertices, when the number of triangles reaches $\Omega(n^3)$, the graph must be dense, i.e., $e(G) = \Omega(n^2)$. Naturally, we explore the scenario where the number of triangles reaches $\Omega(m^{3/2})$ for a graph with $m$ edges. 
In these graphs, there might be some edges that are in very few triangles. Our goal is to find a dense subgraph that contains a large number of triangles.

Another motivation for our study comes from a classical result of Erd\H{o}s and Simonovits~\cite{Erdos_Intro}, who proved that for any $0 < \epsilon < 1$ and sufficiently large $n$, every $n$-vertex graph with $m \geq n^{1+\epsilon}$ edges contains a subgraph $G'$ of order $n_0\geq n^{\epsilon\frac{1-\epsilon}{1+\epsilon}}$ and size $e(G')\geq \frac{2}{5}n_{0}^{1+\epsilon}$, and satisfies that $\Delta(G')/\delta(G') \leq c_\epsilon$, where $c_\epsilon = 10\cdot 2^{\frac{1}{\epsilon^2}+1}$. This result was later improved by Jiang and Seiver~\cite{Jiang_Intro}. We consider an analogue for $r$-clique counts with a given size. When a graph $G$ with $m$ edges contains at least $Cm^{\frac{\alpha r}{2}}$ $K_r$-copies, we show that $G$ must have a subgraph with $\frac{1}{2}Cm^{\frac{\alpha r}{2}}$ $K_r$-copies while maintaining high edge density. However, we cannot guarantee bounded degree ratios as in the Erd\H{o}s-Simonovits result: Consider the complete tripartite graph with vertex parts of size $n$, $n^{1/2}$, and $n^{1/3}$. It has $m=\Theta(n^{3/2})$ edges and $n^{11/6} = \Theta(m^{11/9})$ triangles, but any subgraph $G'$ with $\Theta(m^{11/9})$ triangles must have $\lim_{n\to\infty}\frac{\Delta(G')}{\delta(G')} = +\infty$.
Motivated by these observations, we obtain the following more general theorem.
\begin{thm}\label{Many_Kr_indicates_subgraph_with_many_Kr}
Let $C>0$, $r \geq 3$ be an integer and $\alpha \in\left(\frac{2}{r}, 1\right]$ be a real number. 
Let $G$ be a graph with $m$ edges.
Suppose $k_r(G) \geq C m^{\frac{\alpha r}{2}}$.
Then there exists a subgraph $G_0 \subseteq G$ of order $n_0\geq C_1m^{\frac{\alpha}{2}}$ such that for $i=2,\ldots,r$, $$k_i(G_0) \geq C_i n_0^{\frac{i(r-2) \alpha}{(2-\alpha) r-2}},$$ where constants $C_1, C_2, \ldots, C_r$ depend only on $C, r$ and $\alpha$.
\end{thm}
It is noteworthy that if $\alpha=1$, Theorem~\ref{Many_Kr_indicates_subgraph_with_many_Kr} shows that if $k_r(G) = \Theta(m^{\frac{r}{2}})$, then the graph contains a dense subgraph $G_0$ of size $n_0=\Theta(m^{\frac{1}{2}})$ with $k_i(G_0) = \Theta(n_0^i)$ for all $2 \leq i \leq r$. If $\alpha<1$, Theorem~\ref{Many_Kr_indicates_subgraph_with_many_Kr} yields the following corollary.
\begin{cor}\label{ex(n,kr,F)_indicates_mex(m,Kr,F)}
    Let $r\geq 3$ be an integer and $F$ be a graph. If $\mathrm{ex}(n,K_r,F)=o(n^s)$ for some $1<s\le r$, then $\mathrm{mex}(m,K_r,F)=o(m^{\frac{\left(r-1\right)s}{r+s-2}})$.
\end{cor}

In Corollary \ref{ex(n,kr,F)_indicates_mex(m,Kr,F)}, we investigate the relationship between $\ex(n,K_r,F)$ and $\mathrm{mex}(m,K_r,F)$. Indeed, previous studies on generalized Turán problems (for example, \cite{Alon_Many T copies in H-free graphs}) mainly use the property that the neighborhood of $v$ is $(F-v)$-free for recursive induction. It is easy to see that for an extremal graph $G$ maximizing $\mathrm{mex}(m,K_r,F)$, $\mathrm{mex}(m,K_r,F)=k_r(G)\le  \frac{1}{r}\sum_{v\in G}\ex(d(v),K_{r-1},F-v)$. Consequently, if $\ex(n, K_{r-1}, F-v) = O(n^s)$, then $\mathrm{mex}(m,K_r,F) = O(m^s)$. However, this upper bound is rough and not tight when $s > 1$. Therefore, we  improve the upper bound by the following theorem.

\begin{thm}\label{Upperbound_of_Kr_in_some_F_free}
     Let $F$ be a graph with at least $3$ vertices and $r\geq 2$ be an integer. If $\mathrm{ex}(n,F)\leq C_1 n^\alpha$ and $\min\limits_{v_0\in F} (\mathrm{ex}(n,K_r,F - v_0))\leq C_2n^\beta$ where $\alpha>1$ and $\beta >1$, then there exists a constant $C=C(C_1,C_2,\alpha,\beta,F)$ such that
     $$
     \mathrm{mex}(m,K_{r+1},F)\le Cm^{f(\alpha,\beta)},
     $$
     where $f(\alpha,\beta)=1+(\beta-1)\left(1-\frac 1 \alpha\right)$.
 \end{thm}

As a consequence, the upper bounds on the generalized Tur\'an number $\mathrm{ex}(n, K_r, F)$ carry over directly to those on $\mathrm{mex}(m, K_r, F)$. For instance, Alon and Shikhelman~\cite{Alon_Many T copies in H-free graphs} showed $\mathrm{ex}(n,K_r,K_{s,t}) = O(n^{r - \frac{r(r-1)}{2s}})$ for fixed $r \geq 2$ and $t \geq s \geq r-1$.
Therefore, Theorem \ref{Upperbound_of_Kr_in_some_F_free} yields the following.

\begin{cor}\label{UpperBound_K_r_in_K_st-free_Fixed_Edge}
    For any fixed $r\ge 3$ and $t\ge s\ge r$,
    $$
    \mathrm{mex}(m,K_r,K_{s,t})=O\left(m^\frac{rs-\binom{r}{2}}{2s-1}\right).
    $$
    In particular, for fixed $r,s\ge 2r-2$ and $t\ge(s-1)!+1$,
    \begin{align}\label{mex_K_r_K_{s,t}}
        \mathrm{mex}(m,K_r,K_{s,t})=\Theta\left(m^\frac{rs-\binom{r}{2}}{2s-1}\right).
    \end{align}
    When $r=3$, (\ref{mex_K_r_K_{s,t}}) also holds for $s\geq 2$ and $t\geq (s-1)!+1$.
\end{cor}

Let $\text{ex}_{u}(t, K_r, F)$ be the maximum number of $K_r$'s in an $F$-free graph $G$ containing exactly $t$ copies of $u$-cliques. This notion was recently introduced by Kirsch in~\cite{Kirsch_25_Intro}, and has been studied in~\cite{15_Aragao_Intro_tku, 14_Kirsch_Intro_Ku, Kirsch_25_Intro, 14_Kirsch_Intro_edge}. In particular, for $u = 1$ and $u = 2$, this notation coincides with $\text{ex}(n, K_r, F)$ and $\text{mex}(m, K_r, F)$, respectively. We now derive a general lower bound for the number of $K_r$'s in a $F$-free graph.

\begin{thm}\label{Lowerbound_of_Kr_in_some_F_free}
     Let $r>u\ge2$ and  $F=(V(F),E(F))$ with $v(F)> 2$ and $e(F)>\frac{r-1}{2}v(F)+\binom r 2-(r-1)$ and $\max\limits_{F_0\subseteq F,e(F_0)>0}\frac{2e(F_0)}{v(F_0)}<\frac{2e(F)-r(r-1)}{v(F)-2}$. 
     For sufficiently large $t$, we have
     $$
     \mathrm{ex}_u(t,K_r,F)\geq C_{u,r}k_u(G)^{\frac{re(F)-\binom r 2 v(F)-r\binom r 2+2\binom r 2}{ue(F)-\binom u 2 v(F)-u\binom r 2+2\binom u 2}  },
     $$
     where $C_{u,r}$ is a constant depending on $F,u$ and $r$.
\end{thm}

The rest of the paper is organized as follows. In Section 2, we show some useful lemmas. We prove Theorems~\ref{Many_Kr_indicates_subgraph_with_many_Kr}, \ref{Upperbound_of_Kr_in_some_F_free} and \ref{Lowerbound_of_Kr_in_some_F_free}, along with some corollaries in Section 3. In Section 4, we show the applications of these theorems to some widely studied graph classes, such as $K_{s_1,\dots,s_r}$-free graphs. Finally, we conclude with remarks and discussions in Section 5.

\section{Preliminaries}

Let $G=(V,E)$ be a graph and $r\geq 2$ be an integer. Suppose that $v\in V$ and $e\in E$. Let $N_G(v)$ be the neighborhood of $v$ in $G$. We denote by $\mathcal{N}_{v}(H,G)$ (respectively, $\mathcal{N}_{e}(H,G)$) the number of $H$-copies that contain $v$ (respectively, $e$) in $G$. Our first lemma, which gives an upper bound of $k_r(G)$ by $k_u(G)$ for $r>u\geq 1$, is a direct generalization of the fact that $m<\frac{1}{2}n^2$ and $k_3(G)< \frac{\sqrt{2}}{3}m^{\frac{3}{2}}$ for any graph $G$ of order $n$ and size $m$.

\begin{lem}\label{Max_Ku_in_G_by_Kr}
Let $G$ be a graph and $r>u \geq 1$. 
\begin{align}\label{K_r<K_u}
k_r(G)<C(u, r)k_u(G)^{\frac{r}{u}},
\end{align} 
where $C(u, r)=\frac{(u!)^\frac r u}{r!}$. Furthermore, $k_r(G)=O(m^{\frac{r}{2}})$ for any graph $G$ of size $m$.
\end{lem}

\begin{proof}
    Note that for each pair of $r>u\geq 1$, $C(u,r)$ is determined by the following recursive formulas: (i) $C(1,2)=\frac{1}{2}$, (ii) $C(u,u+1)=\frac{u}{u+1}C(u-1,u)^{\frac{u-1}{u}}$ for $u\geq 2$, and (iii) $C(u,r+1)=C(r,r+1)C(u,r)^{\frac{r+1}{r}}$. 
    
    First, we prove (\ref{K_r<K_u}) for $r=u+1$ by induction. When $u=1$ and $r=2$, it is easy to see that $$k_2(G)=m<\frac{1}{2}n^2=C(1,2)n^2=C(1,2)k_1(G)^2.$$ Now suppose 
     (\ref{K_r<K_u}) holds for $u=u_0-1,r=u_0$ for some $u_0\geq 2$, then for any $v\in V(G)$, 
     \begin{align}\label{2_1_1}
         \mathcal{N}_{v}(K_{u_0+1},G)=k_{u_0}(G[N_G(v)])\le C(u_0-1,u_0)\mathcal{N}_v(K_{u_0},G)^{\frac{u_0}{u_0-1}}.
     \end{align}
    Moreover, 
    \begin{align}\label{2_1_2}
        \mathcal{N}_v(K_{u_0+1},G)\le k_{u_0}(G\backslash v)\le k_{u_0}(G).
    \end{align}
    By (\ref{2_1_1}) and (\ref{2_1_2}), we have 
    $$
    \begin{aligned}
        \sum\limits_{v\in V(G)}\mathcal{N}_{v}(K_{u_0+1},G)&\le 
         \sum_{v\in V(G)}\left(C(u_0-1,u_0)\mathcal{N}_v(K_{u_0},G)^{\frac{u_0}{u_0-1}}\right)^{1-\frac{1}{u_0}}\cdot k_{u_0}(G)^{\frac 1 {u_0}}
        \\&\leq\sum\limits_{v\in V(G)}\left(C(u_0-1,u_0)^{\frac{u_0-1}{u_0}}\mathcal{N}_v(K_{u_0},G)\cdot k_{u_0}(G)^{\frac {1}{u_0}}\right)\\
        &=C(u_0-1,u_0)^{\frac{u_0-1}{u_0}}\cdot u_0 k_{u_0}(G)\cdot k_{u_0}(G)^{\frac 1 {u_0}}\\
        &=(u_0+1)C(u_0,u_0+1)k_{u_0}(G)^{\frac {u_0+1} {u_0}}.
    \end{aligned}
    $$
    Therefore, $k_{u+1}(G)<C(u,u+1)k_u(G)^{\frac{u+1}{u}}$. And by induction, for any fixed $u$ and $k\ge 1$, by (ii) we have
    $$
    \begin{aligned}
        k_{u+k+1}(G)&<C(u+k,u+k+1)\left(C(u,u+k)k_u(G)^{\frac{u+k}{u}}\right)^{\frac{u+k+1}{u+k}}\\&=C(u,u+k+1)k_u(G)^{\frac{u+k+1}{u}}, 
    \end{aligned}
    $$ 
    where $C(u,u+k+1)=C(u+k,u+k+1)C(u,u+k)^{\frac{u+k+1}{u+k}}$.
\end{proof}

We point out that $C(u, r)$ given in Lemma~\ref{Max_Ku_in_G_by_Kr} is asymptotically tight. In fact, consider the sequence $\{K_l\}$ of complete graphs of size $l$, then we have: $$C(u,r)=\frac{\left(u!\right)^{\frac{r}{u}}}{r!}=\lim_{l\to \infty}\frac{k_r(K_l)}{k_u(K_l)^{\frac{r}{u}}}.$$
In particular, when $u = 1$, the conclusion of Lemma~\ref{Max_Ku_in_G_by_Kr} becomes $k_r(G) \leq C(1, r) n^r$. Alon and Shikhelman~\cite{Alon_Many T copies in H-free graphs} gave the following necessary and sufficient condition for an $F$-free graph $G$ to contain $\Theta(n^r)$ $K_r$-copies.

\begin{lem}~\cite{Alon_Many T copies in H-free graphs}\label{Alon_Kr_O(n^r)_iff_chi(F)>r}
	For any graph $F$, $\mathrm{ex}(n,K_t,F)=\Theta(n^t)$ if and only if $\chi(F)>t$. 
\end{lem}

The following two lemmas will be used in the proof of 
Theorem~\ref{Lowerbound_of_Kr_in_some_F_free}. Let $X$ be a random variable. We denote its expectation by $E(X)$ and its variance by $\operatorname{Var}(X)$. We say that $0 < p = p(n)$ is bounded away from $1$ 
if $\limsup_{n\to\infty} p(n) < 1 $.

\begin{lem}[Chebyshev's Inequality]\label{Chebyshev}
    Let $X$ be a random variable with finite expectation $\mu = E(X)$ and finite variance $\sigma^2 = \operatorname{Var}(X)$. For any positive real number $k > 0$, we have 
    $$ P(|X - \mu| \geq k \sigma) \leq \frac{1}{k^2}.$$
\end{lem}

\begin{lem}\label{Upperbound_var_H}~\cite{Janson_Random}
Let $H$ be a fixed graph with vertex set $V(H) \subseteq[n]$ and $p=p(n)$ be bounded away from 1. Let $G \sim {G}(n, p)$ and $X_H$ be the number of $H$-copies in $G(n,p)$.
$$
\frac{\operatorname{Var}(X_H)}{E(X_H)^2}=\Theta\left(\frac{1}{\Phi_H}\right),
$$
where, $\Phi_H=\min \left\{\mathbb{E}\left(X_{H_0}\right): H_0 \subseteq H, e_{H_0}>0\right\}=\Theta\left(\min\limits_{H_0\subseteq H,e_{H_0}>0}n^{v(H_0)}p^{e(H_0)}\right)$.
\end{lem}

For a graph $H$, we aim to determine whether $\Phi_H = \Omega(1)$. We define the maximum average degree of the induced graphs of $H$ to be $\underset{H_0 \subseteq H, e_{H_0} > 0}{\max} \frac{2e(H_0)}{v(H_0)}$. Now, we prove that for any complete $r$-partite graph $H$, the maximum average degree is attained by $H$ itself.

\begin{lem}\label{Min_Phi_KPI_2}
    For every positive integer $r \geqslant 2$ and positive integers $1 \le s_1 \leqslant s_2 \leqslant \ldots \leqslant s_r$,
    $$
    \underset{F \subseteq K_{s_1,\ldots,s_r}, v(F)>0}{\max } \frac{2e\left(F\right)}{v\left(F\right)}= \frac{2e\left(K_{s_1,\ldots,s_r}\right)}{v\left(K_{s_1,\ldots,s_r}\right)}.
    $$
\end{lem}

\begin{proof}
    Note that 
    $$
    \underset{F \subseteq K_{s_1,\ldots,s_r}, v(F)>0}{\max } \frac{2e\left(F\right)}{v\left(F\right)}= \underset{x_i\in[s_i] \text{ for } i\in [r]}{\max}\frac{2\sum_{1 \leq i<j \leq r} x_i x_j}{\sum_{i=1}^{r} x_i},
    $$ 
    Let $f(x_1,\ldots,x_r)=\frac{\sum_{1 \leq i<j \leq r} x_i x_j}{\sum_{i=1}^{n} x_i}$. For each $1\leq i\leq r$, we have
    $$\frac{\partial f}{\partial x_i}=\frac{\left(\sum_{i=1}^r x_i\right)\left(\sum_{i=1}^r x_i-x_i\right)-\sum_{1 \leq i<j \leq r} x_i x_j}{\left({\sum_{i=1}^{n} x_i}\right)^2}\ge 0.$$
    Hence, the function $f\left(x_1, \ldots, x_r\right)$ is monotonically non-decreasing with respect to each variable $x_i\ge 1$. Therefore, 
     $$
    \underset{F \subseteq K_{s_1,\ldots,s_r}, e_{F}>0}{\max } \frac{2e\left(F\right)}{v\left(F\right)}\le \underset{x_i \in [s_i] \text{ for } i\in [r]}{\max}2f(x_1,\ldots,x_r)=2f(s_1,\ldots,s_r)=\frac{2e\left(K_{s_1,\ldots,s_r}\right)}{v\left(K_{s_1,\ldots,s_r}\right)},
    $$ 
    which proves the lemma.
\end{proof}

\section{Proof of the main theorems}


In this section, we prove Theorems~\ref{Many_Kr_indicates_subgraph_with_many_Kr},~\ref{Upperbound_of_Kr_in_some_F_free} and~\ref{Lowerbound_of_Kr_in_some_F_free}. In the proof Theorem~\ref{Many_Kr_indicates_subgraph_with_many_Kr}, we first simultaneously remove edges that are not contained in many $K_r$-copies. Then we demonstrate that there are still sufficiently many remaining  edges and remaining $K_r$-copies, while the number of vertices is upper bounded. Let $G = (V, E)$ be a graph and $V'\subseteq V$. We denote by $G[V']$ the subgraph of $G$ induced by the vertex set $V'$.
\begin{proof}[Proof of Theorem \ref{Many_Kr_indicates_subgraph_with_many_Kr}]
    Let $G=(V, E)$ and define $E_1 \subseteq E$ as follows:
    $$
    E_1=\left\{e \in E \left\lvert\, \mathcal{N}_e\left(K_r, G\right) \leq \frac{1}{2} C m^{\frac{\alpha r-2}{2}}\right.\right\},
    $$
    Let $E_2=E \backslash E_1$,
   $V_2=V(E_2)$ and $G_2$ be the subgraph of $G$ induced by the edge set $E_2$. We aim to show that $G_2$ is the desired subgraph.
   
   First we show that $\left|E_2\right| \ge \left(\frac{C}{2 C(2, r)}\right)^{\frac{2}{r}}m^{ \alpha }$, where $C(2,r)$ is defined in Lemma \ref{Max_Ku_in_G_by_Kr}. For otherwise, by Lemma \ref{Max_Ku_in_G_by_Kr}, we have
$$
\begin{aligned}
k_r(G_2)-C(2, r)\left(k_2(G_2)\right)^{\frac{r}{2}} & \geq k_r(G)-\sum_{e \in E_1} \mathcal{N}_e\left(K_r, G\right)-C(2, r)\left(m-\left|E_1\right|\right)^{\frac{r}{2}} \\
 & \geq m^{\frac{\alpha r}{2}}\left(C-\frac{C}{2} \frac{\left|E_1\right|}{m}-C(2, r)m^{\frac{(1-\alpha)r}{2}}\left(1-\frac{\left|E_1\right|}{m}\right)^{\frac{r}{2}}\right) \\
& = m^{\frac{\alpha r}{2}}\left(\frac{C}{2}+\frac C 2\frac{\left|E_2\right|}{m}-C(2,r)\left(\frac{|E_2|}{m^{\alpha}}\right)^\frac{r}{2}\right)\\
& >0,
\end{aligned}
$$
which contradicts to Lemma \ref{Max_Ku_in_G_by_Kr}.

Moreover, we have
$$k_r(G_2)\ge k_r(G)-\sum\limits_{e\in E_1}\mathcal{N}_e(K_r,G)>\frac{C}2m^{\frac {\alpha r} 2},$$ and $$|V_2|>\sqrt{2|E_2|}\ge \sqrt{2} \left(\frac{C}{2 C(2, r)}\right)^{\frac{1}{r}}m^{\frac{\alpha}{2}}.$$

Now we show an upper bound of $\left|V_2\right|$ with regard to $m$. For each $x\in V_2$, there exists an edge $e = xy \in E_2$ incident to $x$. Then we have $$\begin{aligned}
    \mathcal{N}_{e}\left(K_r,G\right)&\leq k_{r-2}\left(G[N_{G}\left(x\right)\cap N_{G}\left(y\right)]\right) \leq k_{r-2}\left(K_{\left|N_{G}\left(x\right)\cap N_{G}\left(y\right)\right|}\right)\\&\leq k_{r-2}\left(K_{d_{G}\left(x\right)}\right) =\binom{d_G\left(x\right)}{r-2}\leq \frac{d_G\left(x\right)^{r-2}}{\left(r-2\right)!},
\end{aligned}$$
which implies that $$d_G\left(x\right)\geq\left((r-2) !N_e\left(K_r, G\right)\right)^{\frac{1}{r-2}}>\left(\frac{\left(r-2\right)!}{2}Cm^{\frac{\alpha r-2}{2}}\right)^{\frac{1}{r-2}}.$$

Therefore, 
$$
2m=\sum_{x \in V} d_G(x) \geq \sum_{x \in V_2} d_G(x)>\left(\frac{\left(r-2\right)!}{2} C\right)^{\frac{1}{r-2}} m^{\frac{\alpha r-2}{2(r-2)}}\left|V_2\right|.
$$
Hence, we have $$|V_2|\le C_0 m^\frac{(2-\alpha)r-2}{2(r-2)},$$ and $$k_r(G_2)\ge\frac {C}{2} m^\frac{\alpha r}{2}\ge C_r|V_2|^{\frac{r\left(r-2\right)\alpha}{\left(2-\alpha\right)r-2}},$$ where $C_0=2\left(\frac{\left(r-2\right)!}{2} C\right)^{-\frac{1}{r-2}}$ and $C_r=\frac{C}{2} C_0^{-\frac{\alpha r}{2}\cdot\frac{(2-\alpha)r-2}{2(r-2)}}$, both of which depend only on $C,r$ and $\alpha$ as desired.

Finally, for any $2\le i< r$, by Lemma~\ref{Max_Ku_in_G_by_Kr},
$$
k_i(G_2)>\left(\frac{k_r(G_2)}{C(i,r)}\right)^\frac{i}{r}=\left(\frac{C_r}{C(i,r)}\right)^\frac{i}{r} |V_2|^{\frac{i\left(r-2\right)\alpha}{\left(2-\alpha\right)r-2}},
$$
which proves the theorem.
\end{proof}
By applying Theorem~\ref{Many_Kr_indicates_subgraph_with_many_Kr}, we can directly derive Corollary \ref{ex(n,kr,F)_indicates_mex(m,Kr,F)}.

\begin{proof}[Proof of Corollary \ref{ex(n,kr,F)_indicates_mex(m,Kr,F)}]
    To prove $\mathrm{mex}(m,K_r,F)=o(m^{\frac{\left(r-1\right)s}{r+s-2}})$, suppose, for contradiction, that there exists an constant $C$ that for any $M$, there exists an $F$-free graph $G$ of size $m\ge M$ that $k_r(G)\ge Cm^{\frac{\left(r-1\right)s}{r+s-2}}$. By Theorem \ref{Many_Kr_indicates_subgraph_with_many_Kr}, since $k_r(G)\ge Cm^{\frac{\left(r-1\right)s}{r+s-2}}$ and $\frac{2\left(r-1\right)s}{r(r+s-2)}\in(\frac{2}{r},1]$, $G$ contains a subgraph $H$ of order $n_0\ge C_1(C,r,s) m^{\frac{\left(r-1\right)s}{r(r+s-2)}}$ with $k_r(H)\ge C_2(C,r,s) n_0^s$, which contradicts to $\mathrm{ex}(n,K_r,F)=o(n^s)$.
\end{proof}

In particular, when $\alpha=1$, we can derive the following corollary from the upper bound of $|V_2|$ obtained in the proof of Theorem~\ref{Many_Kr_indicates_subgraph_with_many_Kr}.

\begin{cor}\label{Many_Kr_Indicates_Dense_m_Subgraph}
	For any fixed $C>0$ and integer $r\ge3$, there exists constants $C_1=C_1(C,r)$ and $C_2=C_2(C,r)$ such that for any graph $G$ of size $m$ with $k_r(G)\geq Cm^{\frac{r}{2}}$ has a subgraph $H$ of order $n_0$ and size $m_0$, where $n_0\geq C_1m^{\frac{1}{2}}$, $m_0\geq C_2n_0^2$ and $k_r(H)> \frac C 2m^{\frac{r}{2}}$. 
\end{cor}

Combining Corollary~\ref{ex(n,kr,F)_indicates_mex(m,Kr,F)} and Lemma~\ref{Alon_Kr_O(n^r)_iff_chi(F)>r} obtained by Alon and Shikhelman in~\cite{Alon_Many T copies in H-free graphs}, we can derive a similar corollary given the number of edges $m$. 

\begin{cor} \label{When_number_Kt_reach_m^{t/2}}
	For any graph $F$, $\mathrm{mex}(m,K_t,F)=\Omega(m^{\frac{t}{2}})$ if and only if $\chi(F)>t$. 
\end{cor}

Now we sketch the proof of Theorem~\ref{Upperbound_of_Kr_in_some_F_free}. First, assume for contradiction that a minimal counterexample exists. Then we show that every vertex in this counterexample has large degree, which leads to a contradiction to the bound on $\text{ex}(n, F)$.

\begin{proof}[Proof of Theorem \ref{Upperbound_of_Kr_in_some_F_free}]
    Define $$C'=\max\left\{\frac{\mathcal{N}\left(K_{r+1},G\right)}{e(G)^f}:G\text{ is a non-empty graph with }v(G)\leq v(F)-1\right\},$$
    where $f=f(\alpha,\beta)$. 
    Let $C$ be a constant  satisfying $C>\max(C',C_1^{\frac{\beta - 1}{\alpha}}C_2)$ which depends only on $C_1,C_2,\alpha,\beta$ and $F$.  Now we prove $\mathrm{mex}(m,K_{r+1},F)\leq Cm^f$ by contradiction. Suppose that $n$ is the minimum positive integer such that there exists a $F$-free graph $G_0$ of order $n$ satisfying $\mathcal{N}(K_{r+1},G_0)>Cm^f$ where $m = e(G_0)$. By the definition of $C$,  we have $n\geq v(F)$.
  Note that for any $v \in V(G_0)$, the number of $K_{r+1}$'s that contains $v$ in $G_0$ is equal to the number of $K_{r}$'s in $G_0[N_v]$. 
  Let $v_0 \in V(F)$ be a vertex that attains the minimum of $\min_{v_0 \in F} \ex(n,K_r,F-v_0)$.
  Since $G_0$ is $F$-free, $G_0[N_v]$ is $(F-v_0)$-free, which implies $\mathcal{N}_v(K_{r+1},G_0)\le C_2d_{G_0}(v)^\beta$. Therefore, we have   
    \begin{align}
        C(m-d_{G_0}(v))^f\ge k_{r+1}(G_0-v) = k_{r+1}(G_0)-k_{r}(G_0[N_{G_0}(v)]) > Cm^f-C_2d_{G_0}(v)^\beta.
    \end{align}
    Hence,
    \begin{align}
        C_2d_{G_0}(v)^\beta\ge Cm^f\left(1-\left(1-\frac{d_{G_0}(v)}{m}\right)^f\right)\ge Cm^{f-1}d_{G_0}(v).
    \end{align}
    Since $\beta>1$, we obtain for each $v\in V(G),$ $d_{G_0}(v)\geq\left(\frac{C}{C_2}m^{f-1}\right)^\frac{1}{\beta-1}$. Thus, $$2m=\sum_{v\in V\left(G_0\right)}d_{G_0}\left(v\right)\geq n\cdot \left(\frac{C}{C_2}m^{f-1} \right)^\frac{1}{\beta-1},$$ which implies that $$n\leq 2\left(\frac{C_2}{C}\right)^{\frac{1}{\beta - 1}}m^{\frac 1 \alpha}<C_1^{-\frac{1}{\alpha}}m^{\frac 1 \alpha}.$$  This contradicts to the fact that $m\le C_1n^\alpha$ since $G_0$ is $F$-free.
\end{proof}

We are ready to prove a lower bound of $ex_u(t,K_r,F)$ by the probabilistic method.
\begin{proof}[Proof of Theorem \ref{Lowerbound_of_Kr_in_some_F_free}]
It is clear that there exists a positive constant $C'$ such that $\frac{(C')^{\binom{r}{2}}}{2r!}-\frac 3 2(r-2)! \left(2C'\right)^{e(F)}>0$. Let $C_0=2\frac{(2C')^{\binom{u}{2}}}{u!}$, $C_u=\frac{C'}{C_0}$, and $n$ be an integer that $$C_0 n^{u-\frac{v(F)-2}{e(F)-\binom {r} {2}}\binom {u} {2}}<t\le 2 C_0n^{u-\frac{v(F)-2}{e(F)-\binom {r} {2}}\binom {u} {2}}.$$   

Let $G$ be the Erd\H{o}s-R\'enyi random graph $G(n,p)$ with
$$
p=C_u\frac {t} {{n}^{u-\frac{v(F)-2}{e(F)-\binom r 2}\left(\binom u 2-1\right)}}=\Theta\left(n^{-\frac {v(F)-2} {e(F)-\binom r 2}}\right).
$$
Hence, we have
$$
\begin{aligned}
E(k_u(G))&=\binom n u \left(\frac {C_ut} {n^{u-\frac{v(F)-2}{e(F)-\binom r 2}\left(\binom u 2-1\right)}}\right)^{\binom{u}{2}}\\
&<\frac{n^u}{u!}C_u^{\binom{u}{2}}\left(\frac{2C_0 n^{u-\frac{v(F)-2}{e(F)-\binom r 2}\binom u 2}}{n^{u-\frac{v(F)-2}{e(F)-\binom r 2}\left(\binom u 2-1\right)}}\right)^{\binom{u}{2}}\\
&=\frac{(2C_0C_u)^{\binom{u}{2}}}{u!}n^{u-\frac{v(F)-2}{e(F)-\binom r 2}\binom u 2}\\&\le\frac t 2.
\end{aligned}
$$
By Chebyshev inequality, $$P(k_u(G)\le t)\ge P\left(|k_u(G)-E(k_u(G))|\le E(k_u(G))\right)\ge1-\frac{\operatorname{Var}(k_u(G))}{E(k_u(G))^2}=1-O\left(\frac 1 {\Phi(K_u)}\right),$$ where 
$$
\begin{aligned}
\Phi_{K_u}&=\Theta\left(\min\limits_{H_0\subseteq K_u,e_{H_0}>0}n^{v(H_0)}p^{e(H_0)}\right)=\Theta\left(\min\limits_{v(H_0)\le u< r}n^{v(H_0)}\left(n^{-\frac {v(F)-2} {e(F)-\binom r 2}}\ \right)^{\binom{v(H_0)}{2}}\right)\\
&\ge\Theta\left(\min\limits_{v(H_0)\le r}\left(n^{v(H_0)}\right)^{1- \frac {\frac{V(H_0)-1}2(v(F)-2)} {e(F)-\binom r 2}}\right)\ge
\Theta\left(\min\limits_{v(H_0)\le r}\left(n^{v(H_0)}\right)^{1- \frac {\frac{r-1}2v(F)-(r-1)} {e(F)-\binom r 2}}\right).
\end{aligned}
$$
Since $e(F)-\binom r 2>\frac{r-1}{2}v(F)-(r-1)$, as $n\rightarrow \infty$ (with $t\rightarrow \infty$), the inequality $k_u(G)<t$ holds with high probability. Similarly, $E(k_r(G))=\binom n r \left(\frac {C_u t} {n^{u-\frac{v(F)-2}{e(F)-\binom r 2}\left(\binom u 2-1\right)}}\right)^{\binom{r}{2}}$
and $ P\left(|k_r(G)-E(k_r(G))|> {E(k_r(G))}/{2}\right)\le O(\frac{1}{\Phi(K_r)})=o(1)$.
Let $X$ denote the number of $F$-copies in $G$. Then our choice of $p$ ensures that 
\begin{align}\label{E_X}
    E(X)\leq p^{e(F)}n^{v(F)}
\le \left(2C_0C_un^{-\frac {v(F)-2} {e(F)-\binom r 2}}\right)^{e(F)}n^{v(F)}
\le \left(2C_0C_u\right)^{e(F)}n^{\frac{2e(F)-\binom r 2 v(F)}{e(F)-\binom r 2}}
\end{align}
By Chebyshev inequality and Lemma \ref{Upperbound_var_H},    
$$
\begin{aligned}
        P\left(|X-E(X)|>\frac {E(X)}{2}\right)\le4\frac {\operatorname{Var}(X)}{E(X^2)}=O\left(\frac 1{\Phi_{F}}\right),
\end{aligned}
$$
where $\Phi_{F}=\Theta\left(\min\limits_{F_0\subseteq F,e_{F_0}>0}n^{v(F_0)}p^{e(F_0)}\right)=\Theta\left(\min\limits_{F_0\subseteq F,e_{F_0}>0}n^{e(F_0)\left(\frac{v(F_0)}{e(F_0)}-\frac{v(F)-2}{e(F)-\binom{r}{2}}\right)}\right)$. Therefore, since $\max\limits_{F_0\subseteq F,e(F_0)>0}\frac{2e(F_0)}{v(F_0)}<\frac{2e(F)-r(r-1)}{v(F)-2}$, that is $\min\limits_{F_0\subseteq F,e_{F_0}>0}\frac{v(F_0)}{e(F_0)}>\frac{v(F)-2}{e(F)-\binom{r}{2}}$, $O\left(\frac 1{\Phi_{F}}\right)=o(1)$, the inequality $X<\frac{3E(X)}{2}$ holds with high probability as $n\rightarrow \infty$.

On the other hand, deleting an edge destroys at most $\binom {n-2} {r-2}$ $K_r$'s. Note that  
\begin{align}\label{E_K_r}
    E(k_r(G))=\binom n r \left(\frac {C_ut} {n^{u-\frac{v(F)-2}{e(F)-\binom r 2}\left(\binom u 2-1\right)}}\right)^{\binom{r}{2}}>\frac{(C_0C_u)^{\binom{r}{2}}}{r!}
n^{r-\frac{v(F)-2}{e(F)-\binom r 2}\binom r 2},
\end{align}

 and $P(k_r(G)>E(k_r(G))/2)=1-o(1)$.
 Hence by (\ref{E_X}) and (\ref{E_K_r})
$$
\begin{aligned}
        &P\left(k_r(G)-\binom {n-2} {r-2}X>\left(\frac{(C_0C_u)^{\binom{r}{2}}}{2r!}-\frac 3 2(r-2)! \left(2C_0C_u\right)^{e(F)}\right)n^{r-\frac{v(F)-2}{e(F)-\binom r 2}\binom r 2} \right)\\
        \ge&P\left(\left(k_r(G)>\frac{(C_0C_u)^{\binom{r}{2}}}{2r!}n^{r-\frac{v(F)-2}{e(F)-\binom r 2}\binom r 2}\right)\bigcap \left(X<\frac 3 2 \left(2C_0C_u\right)^{e(F)}n^{\frac{2e(F)-\binom r 2 v(F)}{e(F)-\binom r 2}} \right)\right)
        \\
        \ge&P\left(k_r(G)>\frac{E(k_r(G))}{2}\right)+P\left(X<\frac 3 2E(X)\right)-1
        \\
        =&1-o(1).
    \end{aligned}
$$
So with high probability, there exists a $F$-free graph $G'$ with $$k_r(G')>\varepsilon n^{r-\frac{v(F)-2}{e(F)-\binom r 2}\binom r 2}=\Theta(t^{u-\frac{v(F)-2}{e(F)-\binom r 2}\binom u 2}),$$ and $$k_u(G')<k_u(G)<2E(k_u(G))=t.$$ Therefore, for $F$-free graph $G$ with $k_u(G)=t$, $$k_r(G)= \Omega\left(t^{\frac{re(F)-\binom r 2 v(F)-r\binom r 2+2\binom r 2}{ue(F)-\binom u 2 v(F)-u\binom r 2+2\binom u 2}}\right)$$ holds with high probability, which proves the theorem.		
\end{proof}

\section{Applications}
 Applying the theorems in the previous section, we can obtain upper and lower bounds of $\mathrm{mex}(m,K_r,F)$ for various families of graphs.
\subsection{$K_{s,t}$-free graphs}

In this subsection, we prove Corollary \ref{UpperBound_K_r_in_K_st-free_Fixed_Edge}. 
The following lemma establishes an upper bound on the maximum number of $K_r$'s in a $K_{s,t}$-free graph with $n$ vertices.
\begin{lem}\label{Upper_n_Kr}~\cite{Alon_Many T copies in H-free graphs}
    For any fixed $r \geq 2$ and $t \geq s \geq r-1$,
$$
\mathrm{ex}\left(n, K_r, K_{s, t}\right) \leq\left(\frac{1}{r!}+o(1)\right)(t-1)^{\frac{r(r-1)}{2 s}} n^{r-\frac{r(r-1)}{2 s}}.
$$

\end{lem}

The projective norm-graphs $H(q, s)$ constructed in~\cite{Alon_Norm_graph} are known to be $K_{s,(s-1)!+1}$-free, making them a natural candidate for studying the behavior of $K_r$'s in graphs that avoid certain complete bipartite subgraphs. The next lemma shows that the number of $K_r$'s in $H(q, s)$ matches the upper bound established in Lemma \ref{Upper_n_Kr} up to a constant factor, demonstrating that the upper bound is asymptotically tight for $s\geq 2r-2$ and $t\geq (s-1)!+1$.

\begin{lem}\label{Lower_n_Kr}~\cite{Alon_Many T copies in H-free graphs}
    For any fixed $r, s \geq 2 r-2$ and $t \geq(s-1)!+1$, let $H=H(q,s)$ be a projective norm-graph and let $n=v(H)$,
$$
\mathcal{N} (K_r, H(q,s)) =\left(\frac{1}{r!}+o(1)\right) n^{r-\frac{r(r-1)}{2 s}}.
$$
In particular, for any fixed $s \geq 2$ and $t \geq(s-1)!+1$, $$\mathcal{N}\left(K_3, H(q,s)\right)=\Theta\left(n^{3-3 / s}\right).
$$
\end{lem}

Now, we can apply Theorem \ref{Upperbound_of_Kr_in_some_F_free} to get $\mathrm{mex}(m,K_r,K_{s,t})$.

\begin{proof}[Proof of Corollary \ref{UpperBound_K_r_in_K_st-free_Fixed_Edge}]
    By the classical K\"{o}vari-S\'{o}s-Tur\'{a}n Theorem, we have $\mathrm{ex}(n,K_{s,t})=O(n^{2-1/s})$.
    By Lemma \ref{Upper_n_Kr}, for $r\geq 3$ and $t\geq s\geq r$, $$\mathrm{ex}(n,K_{r-1},K_{s-1,t})=O\left(n^{r-1-\frac{(r-1)(r-2)}{2(s-1)}}\right).$$ Hence, by Theorem \ref{Upperbound_of_Kr_in_some_F_free}, $$ \mathrm{mex}(m,K_r,K_{s,t})=O\left(m^\frac{rs-\binom{r}{2}}{2s-1}\right).$$ In particular, when $r\geq 4$, $s\geq 2r-2$ and $t\geq (s-1)!+1$, or $r=3$ and $t\geq (s-1)!+1\geq 2$, the existence of the projective norm-graphs shows that 
$ \mathrm{mex}(m,K_r,K_{s,t})=\Theta\left(m^\frac{rs-\binom{r}{2}}{2s-1}\right)$.
\end{proof}

Note that when $r=3$, we obtain that $\text{mex}(m,K_3,K_{s,t})=\Theta(m^{\frac{3s-3}{2s-1}})$ for $s\geq 2$ and $t\geq (s-1)!+1$, and this conclusion can be derived by Theorem \ref{Upperbound_of_Kr_in_some_F_free} and K\"{o}vari-S\'{o}s-Tur\'{a}n Theorem (without applying Lemma~\ref{Upper_n_Kr}).

Moreover, for the general case, we can get a lower bound by Theorem~\ref{Lowerbound_of_Kr_in_some_F_free}.
\begin{thm}\label{Lower_K_r_in_K_{s,t}-free}
    Let $r>u\ge2$ and $t\ge s\ge \max\left(\binom{r}{2},2r-2\right)$. Then there exists $C_{u,r}$ such that
    $$
    ex_u(p,K_r,K_{s,t})\ge C_{u,r}p^{\frac{2rst-r(r-1)(s+t)-r(r-1)(r-2)}{2ust-u(u-1)(s+t)-ur(r-1)+2u(u-1)}}.
    $$
    holds for sufficiently large $p$.
\end{thm}
\begin{proof}
Notice that when $t\geq s\ge2(r-2)$ and $r\geq 3$,
$$
\begin{aligned}
    e(K_{s,t})-\frac{r-1}{2}v(K_{s,t})-\binom r 2+(r-1)&\geq 
\left(s-\frac{r-1}{2}\right)\left(t-\frac{r-1}{2}\right)-\frac{(r-1)(3r-5)}{4}\\ &\geq \frac{1}{4}(6r^2-10r-4)>0, 
\end{aligned}
$$
and since $\frac{2st}{s+t}\ge \min(s,t)\ge \binom{r}{2}$,
$$
\max\limits_{F\subseteq K_{s,t},e_{K_{s,t}}>0}\frac{2e(F)}{v(F)}\le\max_{s_0\le s,t_0\le t}\frac{2s_0t_0}{s_0+t_0}\le \frac{2st}{s+t}<\frac{2st-r(r-1)}{s+t-2}.
$$
Thus, the theorem follows by Theorem~\ref{Lowerbound_of_Kr_in_some_F_free}.
\end{proof}

Set $p=2$ and $s=t\geq 4$, we have the following  corollary by Theorems~\ref{UpperBound_K_r_in_K_st-free_Fixed_Edge} and \ref{Lower_K_r_in_K_{s,t}-free} directly.
\begin{cor}\label{Cor_mex(m,K3,Kst)}
    Let $s\geq 4$ be an integer. Then $$\text{mex}(m,K_3,K_{s,s}) = \Omega(m^{\frac{3s-6}{2s-2}}) \text{ and mex}(m,K_3,K_{s,s})=O(m^{\frac{3s-3}{2s-1}}).$$
\end{cor}

\subsection{$K_{s_1,\ldots,s_r}$-free graph}

Recently, Balogh, Jiang and Luo~\cite{complete_r_partite_free} derived an upper bound of $\mathrm{ex}(n,K_{r},K_{s_1,\cdots,s_r})$ for $r\geq 3$.
\begin{lem}~\cite{complete_r_partite_free}\label{complete_r_partite_free}
     For every positive integer $r \geqslant 3$ and positive integers $s_1 \leqslant s_2 \leqslant \ldots \leqslant s_r$, $ex\left(n, K_r, K_{s_1, s_2, \ldots, s_r}\right)=o\left(n^{r-1 / \prod_{i=1}^{r-1} s_i}\right)$.

\end{lem}
By Corollary~\ref{ex(n,kr,F)_indicates_mex(m,Kr,F)} and Lemma~\ref{complete_r_partite_free}, we obtain an upper bound for $mex\left(m, K_r, K_{s_1, s_2, \ldots, s_r}\right)$.
\begin{thm}\label{m_complete_r_partite_free}
    For every positive integer $r \geqslant 3$ and positive integers $s_1 \leqslant s_2 \leqslant \ldots \leqslant s_r$, $mex\left(m, K_r, K_{s_1, s_2, \ldots, s_r}\right)=o\left(m^{\frac{(r-1)s}{r+s-2}}\right)$, where $s=r-\frac{1}{\prod_{i=1}^{r-1} s_i}$.
\end{thm}

For the case when $r=3$, we have the following.

\begin{cor}
    For positive integers $s_1 \leqslant s_2\leqslant s_3$, $\mathrm{mex}(m,K_3,K_{s_1,s_2,s_3})=o(m^{\frac{3}{2}-\frac{1}{8s_1s_2-2}})$.
    When $\frac{s_1s_2+s_2s_3+s_3s_1}{s_1+s_2+s_3}>\frac{3}{2}$, then $\mathrm{mex}(m,K_3,K_{s_1,s_2,s_3})=\Omega(m^{\frac{3}{2}-\frac{3(s_1+s_2+s_3)-6}{4(s_1s_2+s_2s_3+s_3s_1)-2(s_1+s_2+s_3)-8}})$.
\end{cor}

\begin{proof}
    The upper bound can be obtained by Theorem \ref{m_complete_r_partite_free} directly. The proof of the lower bound requires Theorem~\ref{Lowerbound_of_Kr_in_some_F_free}, so we first verify the conditions are satisfied: $e(K_{s_1,s_2,s_3})=s_1s_2+s_2s_3+s_3s_1>\frac 3 2(s_1+s_2+s_3)>v(K_{s_1,s_2,s_3})+1$;
\begin{align*}
    \max\limits_{F\subseteq K_{s_1,s_2,s_3},e_{K_{s_1,s_2,s_3}}>0}\frac{2e(F)}{v(F)}
    &\ge \frac{2(s_1s_2+s_2s_3+s_3s_1)}{s_1+s_2+s_3}  
    \\
    &>\frac{2(s_1s_2+s_2s_3+s_3s_1)-6}{s_1+s_2+s_3-2}, 
\end{align*}
where the first inequality holds because of Lemma~\ref{Min_Phi_KPI_2} and the second inequality holds as $ \frac{s_1s_2+s_2s_3+s_3s_1}{s_1+s_2+s_3}>\frac{3}{2}$. Hence, by Theorem~\ref{Lowerbound_of_Kr_in_some_F_free}, we have $$\mathrm{mex}(m,K_3,K_{s_1,s_2,s_3})=\Omega(m^{\frac{3}{2}-\frac{3(s_1+s_2+s_3)-6}{4(s_1s_2+s_2s_3+s_3s_1)-2(s_1+s_2+s_3)-8}}).$$
\end{proof}

For two given graphs $G$ and $H$, we denote by $G \vee H$ the graph with vertex set $V(G) \cup V(H)$ and edge set $E(G) \cup E(H) \cup \{xy : x \in V(G), y \in V(H)\}$. In Theorem \ref{m_complete_r_partite_free}, when $s_1=1$, we can get a better upper bound by Theorem \ref{Upperbound_of_Kr_in_some_F_free}. Since $K_{1,s_2,\ldots s_r}=K_1\vee K_{s_2,\ldots,s_r}$, there exists $c_2$ that $\mathrm{ex}(n,K_{r-1},K_{s_2\ldots,s_r})\le c_2 n^{r-1-1/\prod_{i=2}^{r-1}}$. Therefore, by Theorem \ref{Upperbound_of_Kr_in_some_F_free}, $\mathrm{mex}(n,K_{r},K_{1,s_2\ldots,s_r})=O(m^{(r-1/\prod_{i=2}^{r-1})/2})$.

\subsection{$(K_{s}\vee C_{l})$-free graph}

We denote by $C_l$ the cycle on $l$ vertices. The study of the edge count of $C_{2k}$-free graphs has been a significant topic in extremal graph theory. Bondy and Simonovits~\cite{Bondy_C2k} proved that $\text{ex}(n, C_{2k}) = O(n^{1 + 1/k})$, and this result was improved by Bukh and Jiang~\cite{Bukh_bound} to $\text{ex}(n, C_{2k}) \leq 80 \sqrt{k} \log k \cdot n^{1 + 1/k} + O(n)$. Meanwhile, the generalized Tur\'an problem $\text{ex}(n, K_r, C_l)$ has also been widely discussed. Alon and Shikhelman~\cite{Alon_Many T copies in H-free graphs} showed that $\text{ex}(n, K_3, C_{2k+1}) = O(n^{1 + 1/k})$, and this result was extended by Gerbner, Methuku and Vizer~\cite{Gerbner_kF_free} in the following lemma.
\begin{lem}\label{ex(n,Kr,Cl)}~\cite{Gerbner_kF_free}
We have 
\begin{itemize}
    \item[(a)] For any $r \geq 3,  l \geq 2$, 
$\mathrm{ex}\left(n, K_r, C_{2 l+1}\right) =O\left(n^{1+1 / l}\right)$.

\item[(b)] For any $r\ge 2, l \geq 2$, 
$\mathrm{ex}\left(n, K_r, C_{2 l}\right) =O\left(n^{1+1 / l}\right)$.
\end{itemize}
\end{lem}
Now we obtain the following estimates for $\text{mex}(m, K_r, K_{s}\vee C_{l})$.

\begin{thm}\label{mex(m,Kr,KsvCl)}
    Let $r\geq 3$, $l\geq 4$ and $s\geq 1$ be positive integers. Then we have
    \begin{itemize}
        \item[(i)] If $l=2t$ and $r\geq s+2$, or $l=2t+1$ and $r\geq s+3$ for some $t\geq 2$, then $$\mathrm{mex}(m,K_r,K_{s}\vee C_{l})=O(m^{\frac{s+1}{2}+\frac{1}{2t}}).$$
    \item[(ii)] Otherwise, $\text{mex}(m,K_r,K_{s}\vee C_{l})=\Theta(m^{\frac{r}{2}})$.
    \end{itemize}
\end{thm}
\begin{proof}
    If $l=2t$ and $r\leq s+1$, or $l=2t+1$ and $r\leq s+2$ for some $t\geq 2$, then $$\chi(K_{s}\vee C_{l})=\chi(K_s)+\chi(C_{l})\geq r+1,$$ which implies $(m,K_r,K_{s}\vee C_{l})=\Theta(m^{\frac{r}{2}})$ by Corollary~\ref{When_number_Kt_reach_m^{t/2}}.

    If $l=2t$ and $r-s\geq 2$ or $l=2t+1$ and $r-s\geq 3$ and let $G$ be a $K_{s-1}\vee C_{l}$-free graph of order $n$. Note that every $K_{r-1}$ in $G$ can be viewed as $K_{s-1} \vee K_{r-s}$. Since $G$ is $(K_{s-1} \vee C_l)$-free, within each copy of $K_{s-1}$, the common neighborhood of its vertices can contain at most $\text{ex}(n, K_{r-s}, C_l)$ copies of $K_{r-s}$. Thus by Lemma~\ref{ex(n,Kr,Cl)}, $$\text{ex}(n,K_{r-1},K_{s-1}\vee C_l)\le \binom{n}{s-1}\text{ex}(n,K_{r-s},C_l)=O(n^{{s}+\frac{1}{t}}).$$ Therefore, combining the fact that $\text{ex}(n,K_s\vee C_l)=\Theta(n^2)$ and Theorem \ref{Upperbound_of_Kr_in_some_F_free}, we have $\text{mex}(m,K_r,K_{s}\vee C_{l})=O(m^{\frac{s+1}{2}+\frac{1}{2t}}).$
\end{proof}

\section{Concluding remarks}
 For a given graph $H$ on $h$ vertices, let $V(H) = [h]$ and $s$ be a positive integer. We say that a graph $G$ is an $s$-blow-up of $H$ if $V(G)$ can be partitioned into $V_1 \cup \dots \cup V_h$, where each $V_i$ is an independent set, and there is an edge between $v_i \in V_i$ and $v_j \in V_j$ if and only if $i$ and $j$ are adjacent in $H$. Lemma~\ref{Alon_Kr_O(n^r)_iff_chi(F)>r} is a direct corollary of the following proposition obtained by Alon and Shikhelman~\cite{Alon_Many T copies in H-free graphs}, which inspires us to find the necessary and sufficient conditions for $\text{mex}(m, K_r, F) = \Omega(m^{\frac{r}{2}})$, i.e. Corollary~\ref{When_number_Kt_reach_m^{t/2}}.
\begin{thm}~\cite{Alon_Many T copies in H-free graphs}\label{ex(n,T,H)_Blowup}
    Let $T$ be a fixed graph on $t$ vertices. Then $\mathrm{ex}(n,T,H)=\Omega(n^t)$ iff $H$ is not a subgraph of a blow-up of $T$. Otherwise, $\mathrm{ex}(n,T,H)\leq n^{t-\epsilon\left(T,H\right)}$ for some $\epsilon(T,H)>0$.
\end{thm}
However, Theorem \ref{ex(n,T,H)_Blowup} does not have an analogue under the constraint of a fixed number of edges. 
One can see this by showing that $\text{mex}(m,K_{1,r},K_{s,t})=\Theta(m^r)$ for $t\geq s\geq 2$. Note that when $s \geq 2$ and $s r \geq t$, the graph $K_{s,t}$ is a subgraph of an $s$-blow-up of $K_{1,r}$.
\begin{prop}
    Let $t\geq s\geq 2 $ and $ r\geq 1$ be fixed integers, then $$\mathrm{mex}(m,K_{1,r},K_{s,t})=\binom{m}{r}.$$
\end{prop}
\begin{proof}
Let $G$ be a $K_{s,t}$-free graph of size $m$. Then we choose any $r$ edges in $G$. These edges can induce at most $1$ copy of $K_{1,r}$, which implies $\mathcal{N}(K_{1,r},G)\leq \binom{m}{r}$. On the other hand, when $G$ is a star $K_{1,m}$ of size $m$, then $\mathcal{N}(K_{1,r},K_{1,m})=\binom{m}{r}$.
\end{proof}
We further investigate the relationship between $\text{mex}(m, H, F)$ and $\text{ex}(n, H, F)$. By Theorem~\ref{UpperBound_K_r_in_K_st-free_Fixed_Edge}, when $r = 3$, $s\geq 3$ and $t \geq (s-1)! + 1$, we obtain $\text{mex}(m, K_r, K_{s,t}) = \Theta(m^{\frac{3s - 3}{2s - 1}})$. Meanwhile, since $\text{ex}(n, K_{s,t}) = \Theta(n^{2 - \frac{1}{s}})$, it follows that $\text{ex}(n, K_r, K_{s,t}) = O(n^{3 - \frac{3}{s}})$. Furthermore, by the existence of $H(q, s)$, we deduce that $\text{ex}(n, K_r, K_{s,t}) = \Theta(n^{3 - \frac{3}{s}})$. To this end, we introduce the following definition:
\begin{definition}
    Let $F$ and $H$ be two fixed simple graphs. Let $f_F(n)=\text{ex}(n,F)$, $f_{H,F}(n)=\text{ex}(n,H,F)$ and $g_{H,F}(m)=\text{mex}(m,H,F)$. We say $H$ is $F$-edge-Tur\'an-good if for any $f_H(n)\le m<f_H(n+1)$, $g_{H,F}(m)=\Theta(f_{H,F}(n))$.
\end{definition}
According to our discussions above, $K_3$ is $K_{s,t}$-edge-Tur\'an-good for $t\geq (s-1)!+1\geq 3$. Evidently, $K_u$ is $K_r$-edge-Tur\'an-good if $u<r$. Here we give an example that $H$ is not $F$-edge-Tur\'an-good. Let $H=K_3$ and $F=C_{2k+1}$. Notice that $f_{C_{2k+1}}(n)=\Omega(n^2)$, $f_{K_3,C_{2k+1}}(n)=O(n^{1+\frac{1}{k}})$ by Alon and Shikhelman~\cite{Alon_Many T copies in H-free graphs}, and $g_{K_3,C_{2k+1}}(m)=\Theta(m)$, 
$$g_{K_3,C_{2k+1}}(f_{C_{2k+1}}(n))=\Omega(n^2)\gg f_{K_3,C_{2k+1}}(n)=O(n^{1+\frac{1}{k}}).$$ Therefore, we propose the following natural question.
\begin{question}
    Which pair of graphs $(H,F)$ satisfies the condition that $H$ is $F$-edge-Tur\'an-good?
\end{question}

\section*{Acknowledgement}
Yan Wang is supported by the National Key R\&D Program of China under Grant No. 2022YFA1006400 and Shanghai Municipal Education Commission (No. 2024AIYB003).
Xiao-Dong Zhang is supported by the National Natural Science Foundation of China (No. 12371354) and the Science and Technology Commission of Shanghai Municipality (No. 22JC1403600),  the Montenegrin-Chinese Science and Technology.

\end{document}